\documentclass[preprint,12pt]{elsarticle}

\usepackage{lipsum} 
\usepackage{hyperref} 
\usepackage{geometry} 
\usepackage{graphicx}
\usepackage{multirow}
\usepackage{amsmath,amssymb,amsfonts}
\usepackage{amsthm}
\usepackage{mathrsfs}
\usepackage[title]{appendix}
\usepackage{xcolor}
\usepackage{textcomp}
\usepackage{manyfoot}
\usepackage{booktabs}
\usepackage{algorithm}
\usepackage{algorithmicx}
\usepackage{algpseudocode}
\usepackage{listings}
\usepackage{cleveref}
\usepackage{hyperref}
\theoremstyle{plain}
\newtheorem{theorem}{Theorem}[section]
\newtheorem{proposition}[theorem]{Proposition}
\newtheorem{corollary}[theorem]{Corollary}
\newtheorem{lemma}[theorem]{Lemma}

\theoremstyle{remark} 
\newtheorem{note}{Note}[section]
\theoremstyle{definition}
\newtheorem{example}{Example}[section]
\newtheorem{remark}{Remark}[section]
\newtheorem{definition}{Definition}[section]

\journal{Nonlinear Analysis}

\begin{document}

\begin{frontmatter}



\title{Existence and uniqueness of mild solutions and evolution operators for a class of non-autonomous conformable fractional semi-linear systems and Their Exact Null Controllability}


\author[inst1]{Dev Prakash Jha}
\ead{devprakash.22@res.iist.ac.in}
\affiliation[inst1]{organization={Department of Mathematics, Indian Institute of Space Science and Technology},
            addressline={Valiamala P.O.}, 
            city={Thiruvananthapuram},
            postcode={695547}, 
            state={Kerala},
            country={India}}
\author[inst1]{Raju K. George}
\ead{george@iist.ac.in}
\begin{abstract}
This paper investigates the controllability of systems governed by conformable fractional order derivatives. It first establishes the existence and uniqueness of evolution operators for non-autonomous fractional-order homogeneous systems, using a suitable initial time defined as the intersection of two specific time intervals. Using the theory of linear evolution operators, Schauder’s fixed-point theorem, and the Banach contraction principle, the study derives a new set of sufficient conditions for the existence and uniqueness of a mild solution to non-autonomous conformable fractional semi-linear systems. Additionally, the paper examines the exact null controllability of abstract systems based on the mild solution. We provide a comprehensive example to demonstrate the applicability of the established  theoretical results.
\end{abstract}





\end{frontmatter}


\section{Introduction}
\label{sec:intro}
Fractional calculus is an extension of integer-order calculus, and it has become a significant area of research because specific dynamical models can be described more accurately using fractional derivatives rather than integer-order derivatives. Initially explored in a purely mathematical context, fractional integrals and derivatives have gained significance in various scientific and engineering disciplines in recent decades. Their applications span fields such as fluid mechanics, viscoelasticity, biology, physics, image processing, entropy theory, and engineering (see \cite{bib17,bib18,bib19,bib20,bib21,bib22,bib23}).\\

The concept of fractional derivatives traces its origins to L'Hospital's work in 1695, where he initially introduced it. Since then, multiple definitions of fractional derivatives have surfaced, with the Riemann-Liouville and Caputo definitions emerging as the most notable. Additional information on these definitions and their unique characteristics can be found in the references; see (\cite{bib24,bib25,bib26}).\\

Khalil et al. \cite{bib1} were the first to describe a conformable fractional derivative for functions $f:[0, \infty) \rightarrow \mathbb{R}$. This derivative is more general than the classical derivative and follows the usual rules of linearity, product, quotient, Rolle's theorem, and the mean value theorem. It also matches the classical definition of Riemann-Liouville and Caputo on polynomials. Later, Abdeljawad \cite{bib2} introduced the idea of left and right fractional derivatives and used the chain rule, integration by parts, Grönwall's inequality, power series expansion, and Laplace transformation on the fractional derivative to show that they were correct.
Some alternative perspectives on conformable derivatives are explored in (see \cite{bib16,bib27,bib28,bib29,bib30}).\\

Semigroup theory is a powerful tool for solving differential equations, particularly for addressing abstract Cauchy problems. We extend this concept to include semilinear and quasilinear evolution equations and inhomogeneous initial value problems. For foundational information on the theory of semigroups of linear operators, we refer to \cite{engel2000one}, \cite{bib34}, and \cite{goldstein2017semigroups}.\\

This paper proposes a new method for defining the initial time as the intersection of two specific time intervals for a class of non-autonomous conformable fractional-order systems. This method helps to find evolution operators, mild solutions, and exact null controllability for semi-linear systems, and these are results similar to the results of ordinary non-autonomous differential equations.\\

 We apply the theory of semi-groups of linear operators to identify evaluation operators and solve the conformable fractional abstract Cauchy problem, further extending this approach to introduce the mild solution of the abstract systems.\\
 \newpage
Finally, we will establish the boundedness result on $(L_{\zeta})^{-1}(N_{\zeta}^{t_2})$ as demonstrated in \cite{bib36} and then utilize this result, along with the theory of linear evolution operators, to study the solvability of exact null controllability of abstract systems.\\

This paper investigates a class of control systems described by nonlinear differential equations in functional analysis. We designate a Hilbert space for the state, denoted as Z. Our analysis focuses on systems represented by the following nonlinear conformable fractional order differential equation:
\begin{equation} \left. \begin{array}{l} \left (T_\alpha x\right) (t)+A(t)x(t)=Bu(t)+F(t,x(t)), \quad \alpha \in(0,1], \\ 
x(\zeta)=x_{0}, \quad  t \in [t_1,t_2], \quad \text{and} \quad \zeta \in [t_0,t_f] \cap [t_1,t_2],    \end{array} \right\} \tag{P}
\label{eq:P}
\end{equation}
where,

\begin{itemize}
  \item $x(t)$ represents the state that lies in $Z$, and $x_0 \in Z$;
  \item The control function $u(\cdot)$ belongs to $L^2([\zeta,t_2];U)$ of admissible control functions, where $U$ is also a real Hilbert space;
   \item $(T_\alpha f)(t)$ is the conformable fractional derivative of $f$ at $t$, and is defined by
    $$
    \left( T_{\alpha} f \right)(t) = \lim_{\epsilon \rightarrow 0} \frac{f\left(t + \epsilon(t)^{1 - \alpha}\right) - f(t)}{\epsilon}, \quad t > 0;
    $$
  \item Let $A(t): Z \rightarrow Z^*$ be a closed linear operator, which may not necessarily be bounded. Take a look at the group of operators $\{A(t)\}_{\frac{t^\alpha}{\alpha} \in [t_0, t_f]}$ that create an almost strong evolution operator $\Psi_\alpha: \mathcal{D} \rightarrow \text{BL}(Z)$. Here, $\mathcal{D} = \{(\frac{s^\alpha}{\alpha}, \frac{t^\alpha}{\alpha}) \in [t_0, t_f] \times [t_0, t_f]: 0 \leq t_0 \leq \frac{s^\alpha}{\alpha} \leq \frac{t^\alpha}{\alpha} \leq t_f\}$, and $\text{BL}(Z)$ is the space of all bounded linear operators $\Psi_\alpha(t, s)$ from $Z$ into itself, with $||\Psi_\alpha(t, s)|| \leq M$ for $t\geq s;$
  \item $B$ is a bounded linear operator from $U$ into $Z$;
  \item The function $F: [\zeta, Z] \times Z \to Z$ is given and satisfies certain assumptions.
\end{itemize}
\begin{note}
In the upcoming section, we modify the system (\ref{eq:P}) in such a way that
$$\frac{t^\alpha}{\alpha} \in [t_0,t_f] \iff t\in [(\alpha t_0)^\frac{1}{\alpha},(\alpha t_f)^\frac{1}{\alpha}] \iff t \in [t_1,t_2],$$ 
where $t_1=(\alpha t_0)^\frac{1}{\alpha}$ and $t_2=(\alpha t_f)^\frac{1}{\alpha}.$ \\
This has consistently occurred because time, denoted by $t$, is a positive real number.
\end{note}
\newpage
 This work has the following objectives:

\begin{itemize}
\item Establish existence and uniqueness of evolution operators;
\item Establish the existence and uniqueness of mild solutions to semi-linear conformable fractional order systems using linear evolution operators;
\item Obtain exact null controllability of non-autonomous systems;
\item We provide an example to illustrate the applicability of the established  theoretical results
\end{itemize}

This article is structured as follows: In \autoref{sec:Preliminaries},  we recall some fundamental definitions and results on the conformable fractional derivatives. \autoref{sec:evolution}, introduces relevant notations, concepts, hypotheses, and critical results about linear evolution operators and exact null controllability. In \autoref{sec:Exact}, we explore the exact null controllability of systems (\ref{eq:P}). Finally, \autoref{sec:Application}, presents an example to demonstrate the application of the results obtained.

\section{Preliminaries}
\label{sec:Preliminaries}
\begin{definition}(Conformable Fractional Derivative \cite{bib1})
    Consider $f:[0, \infty) \rightarrow \mathbb{R}$. The conformable fractional derivative of $f$ with order $\alpha \in(0,1]$ is expressed as
\begin{equation}
T_{\alpha}(f)(t)=\lim _{\epsilon \rightarrow 0} \frac{f\left(t+\epsilon t^{1-\alpha}\right)-f(t)}{\epsilon}, \quad t>0.
\label{eq:P2}
\end{equation}
If the fractional derivative $f$ of order $\alpha$ is conformable, we designate it $f$ as $\alpha$-differential. This is represented as $T_{\alpha}(f)(t)=f^{(\alpha)}(t)$. If $f$ is $\alpha$-differential within the interval $(0, a)$, $a>0,$ and $\lim _{t \rightarrow 0^{+}} f^{(\alpha)}(t)$ exists, then we denote $f^{(\alpha)}(0)=\lim _{t \rightarrow 0^{+}} f^{(\alpha)}(t)$.
\end{definition}
\begin{definition}(Left Conformable FractionalDerivative \cite{bib2})
A function $f:[a, \infty) \rightarrow \mathbb{R}$ with an order $0 < \alpha \leq 1$ has a left conformable fractional derivative that starts at $a$ and is given by
\begin{equation} \left(T_{\alpha}^{a} f\right)(t)=\lim _{\epsilon \rightarrow 0} \frac{f\left(t+\epsilon(t-a)^{1-\alpha}\right)-f(t)}{\epsilon}, \quad t>a. \label{eq:P3} \end{equation}
Let's denote $ \left(T_{\alpha}^{a} f\right)(t) $ as $ f^{(\alpha, a)}(t) $. In the case where $ a=0 $, we denote $ T_{\alpha}^{a} f $ as $ T_{\alpha}f $. If $ \left(T_{\alpha}^{a} f\right)(t) $ exists on the interval $ (a, b) $, $b>a,$ then $ \left(T_{\alpha}^{a} f\right)(a) = \lim_{t \rightarrow a^{+}} \left(T_{\alpha}^{a} f\right)(t) $.
\end{definition}
\begin{definition}(Right Conformable Fractional Derivative \cite{bib2})
The right conformable fractional derivative of order $\alpha \in(0,1]$ terminating at $b$ of function $f$ is defined as
\begin{equation}
\left({ }_{\alpha}^{b} \mathrm{T} f\right)(t)=-\lim _{\epsilon \rightarrow 0} \frac{f\left(t+\epsilon(b-t)^{1-\alpha}\right)-f(t)}{\epsilon}, \quad t<b. 
\label{eq:P4}
\end{equation}
 If $\left({ }_{\alpha}^{b} \mathrm{T} f\right)(t)$ exist on $(a,b), b>a,$ then $\left({ }_{\alpha}^{b} \mathrm{T} f\right)(b)=\lim _{t \rightarrow b^{-}} \left({ }_{\alpha}^{b} \mathrm{T} f\right)(t). $
\end{definition}
\begin{theorem}(\cite{bib1})\label{thm:th1}
    Suppose $\alpha \in (0,1]$ and $f_1$ and $f_2$ are $\alpha$- differentiable at a point $t>a$. Then:
\begin{enumerate}
  \item $T_{\alpha}^{a}(c f_1 + d f_2) = c T_{\alpha}^{a}(f_1) + d T_{\alpha}^{a}(f_2)$ for all $c, d \in \mathbb{R}$. \hspace{1.7cm} (Linearity)

  \item $T_{\alpha}^{a}(\beta)=0$ for all constant functions $f(t)=\beta \quad \forall t$.
  \item $T_{\alpha}^{a}(f_1 f_2)=f_1 T_{\alpha}^{a}(f_2)+f_2 T_{\alpha}^{a}(f_1)$.
  \item $T_{\alpha}^{a}\left(\frac{f_1}{f_2}\right)=\frac{f_2 T_{\alpha}^{a}(f_1)-f_1 T_{\alpha}^{a}(f_2)}{f_2^{2}}$. \hspace{6cm} (Product rule)
  \item If $f_1$ is differentiable, then $T_{\alpha}^{a}(f_1)(t)=(t-a)^{1-\alpha} \frac{d f_1}{d t}(t)$. \hspace{0.8cm} (Quotient rule)
  \item $T_{\alpha}^{a}\left(\frac{(t-a)^{\alpha}}{\alpha}\right)=1$.
\end{enumerate}
\end{theorem}
\begin{definition}(Chain Rule\cite{bib2})
    Let $f$ and $g$ be functions defined on $(a, \infty)$, where $f$ and $g$ are (left) $\alpha$-differentiable with $0 < \alpha \leq 1$. Define $h(t) = f(g(t))$. Then $h(t)$ is (left) $\alpha$-differentiable, and for all $t \neq a$ and $g(t) \neq 0$, and  we have
\begin{equation}
\left(T_{\alpha}^{a} h\right)(t) = \left(T_{\alpha}^{a} f\right)(g(t)) \left(T_{\alpha}^{a} g\right)(t)(g(t) - a)^{\alpha - 1},
\label{eq:P5}
\end{equation}
For $t=a$, we have
\begin{equation}
\left(T_{\alpha}^{a} h\right)(a) = \lim _{t \rightarrow a^{+}} \left(T_{\alpha}^{a} f\right)(g(t)) \left(T_{\alpha}^{a} g\right)(t) (g(t) - a)^{\alpha - 1}.
\label{eq:P6}
\end{equation}
\end{definition}
\begin{definition} ($\alpha -$ conformable integral \cite{bib2})
    \[
I_\alpha^a (f)(t)=\int_a^t f(x) d(x,\alpha) = I_a^1 (t^{\alpha-1} f) = \int_a^t \frac{f(x)}{(x-a)^{1-\alpha}} \, dx ,
\]
where the integral is the usual Riemann improper integral, and \(\alpha \in (0, 1)\).
\end{definition}
\begin{theorem}(Banach Contraction Principle Extension)\label{thm:th2}
   In a complete metric space $(X,d)$,  each contraction map $f:X \rightarrow X$ has a unique fixed point.
\end{theorem}
Bryant, V.W \cite{bib3} further expanded upon the Banach contraction theorem subsequently:
\begin{theorem}( Generalized contraction mapping principle )\label{thm:th3}
  Let (X, d) be a complete metric space, and let $f: X \rightarrow X$ be a mapping such that for some positive integer $n, f^{(n)}$ is the contraction on $X$. Then, $f$ has a unique fixed point.
\end{theorem}
\begin{theorem}(Schauder's Fixed Point Theorem)
    Let $K$ be a closed convex subset of a Banach space $E$. If $F : K \to K$ is continuous and $F(K)$ is relatively compact, then $F$ has a fixed point in $K$.
\end{theorem}
\begin{proposition}(\cite{bib4})\label{prop:pro4}
There exists a matrix function \( U(t) \) that is \(\alpha\)- conformable derivative, and regular. We have:
\[
 (U^{-1})^{(\alpha)}(t) = -U^{-1}(t)U^{(\alpha)}(t)U^{-1}(t).
\]
\end{proposition}
\begin{theorem}(Differentiation Under the Integral Sign \cite{bib5})\label{thm:th5} The rule for conformable fractional differentiation of an
integral is the following:
\begin{align*}
    &T_{\alpha}\left(\int_{a(t)}^{b(t)}  h(t, s) \, d(s,\alpha)\right)(t)\\
    &=\int_{a(t)}^{b(t)} T_{\alpha}\left(h(t, s) \right)(t)d(s,\alpha)+ h(t, b(t)) (T_{\alpha} b)(t) - h(t, a(t)) (T_{\alpha} a)(t).
\end{align*}
\end{theorem}
Recently, Sarikaya et al. \cite{Sarikaya} introduced a new definition for the function of the variable $p$, described as follows:
\begin{equation}\label{eq:Sari_1}
(p)^{\alpha}_{n,k} = (p + \alpha - 1)(p + \alpha - 1 + \alpha k) \cdots (p + \alpha - 1 + (n-1)\alpha k).
\end{equation}
They referred to this as the Pochhammer symbol \((p)^{\alpha}_{n,k}\). By setting \( k \to 1 \) and \(\alpha \to 1\), it reduces to the standard Pochhammer symbol \((p)_n\). The conformable gamma function \(\Gamma^{\alpha}_k\) is defined as:
\begin{equation}\label{eq:Sari_2}
\Gamma^{\alpha}_k(p) = \int_{0}^{\infty} t^{p-1} e^{-\frac{t^{\alpha k}}{\alpha k}} d({t,\alpha}) = \lim_{n \to \infty} \frac{n! \alpha^n k^n (n\alpha k)^{\frac{p+\alpha-1}{\alpha k}-1}}{(p)^{\alpha}_{n,k}}.
\end{equation}
By using this definition, it becomes evident that the following equations exist:
$$ \Gamma(x) = \lim_{(\alpha, k) \to (1, 1)} \Gamma_k^\alpha (x),$$

$$ \Gamma_k^\alpha (p) = (\alpha k)^{\frac{p+\alpha-1}{\alpha k} - 1} \Gamma \left( \frac{p + \alpha - 1}{\alpha k} \right),$$

$$  \Gamma_k^\alpha (p) = (\alpha)^{\frac{p+\alpha-1}{\alpha k} - 1} \Gamma_k \left( \frac{p + \alpha - 1}{\alpha} \right),$$

$$ \Gamma_k^\alpha (x + \alpha k) = (x + \alpha - 1) \Gamma_k (x).$$
The conformable beta function is established with the following definition:

\begin{equation}\label{eq:Sari_3}
B_k^\alpha (x, y) = \frac{1}{\alpha k} \int_0^1 t^{\frac{x}{\alpha k} - 1} (1 - t)^{\frac{y}{\alpha k} - 1} d({t,\alpha}). 
\end{equation}
The following equation relates the conformal beta function to the classical beta function:
\begin{equation}\label{eq:Sari_4}
B_k^\alpha (p, q) = \frac{1}{\alpha} B_k \left( \frac{p + \alpha k (\alpha - 1)}{\alpha}, \frac{q}{\alpha} \right) = \frac{1}{\alpha k} B \left( \frac{p}{\alpha k} + \alpha - 1, \frac{q}{\alpha k} \right),
\end{equation}

and
\begin{equation}\label{eq:Sari_5}
B_k^\alpha (x + \alpha k (1 - \alpha), y) = \frac{\Gamma_k^\alpha (x) \Gamma_k^\alpha (y)}{\Gamma_k^\alpha (x + y + 1 - \alpha)}.
\end{equation}

Furthermore, Sarikaya et al.\ (see \cite{Sarikaya}) explored the properties of the $(\alpha, k)$-gamma and $(\alpha, k)$-beta functions and derived the following properties:

\begin{equation}\label{eq:Sari_6}
\Gamma_k^\alpha (p + n\alpha k) = (p)_{n,k}^\alpha \Gamma_k^\alpha (p),
\end{equation}

\begin{equation}\label{eq:Sari_7}
\Gamma_k^\alpha (\alpha k + 1 - \alpha) = 1,
\end{equation}

\begin{equation}\label{eq:Sari_8}
\Gamma_k^\alpha (p) = a^{\frac{p+\alpha-1}{\alpha k}} \int_0^\infty t^{p-1} e^{-a\frac{t^{\alpha k}}{\alpha k}} d(t,\alpha).
\end{equation}

and
\begin{equation}\label{eq:Sari_9}
B_k^\alpha (p, \alpha k) = \frac{1}{p + \alpha k (\alpha - 1)},
\end{equation}

\begin{equation}\label{eq:Sari_10}
B_k^\alpha (\alpha k (2 - \alpha), q) = \frac{1}{q},
\end{equation}

\begin{equation}\label{eq:Sari_11}
B_k^\alpha (\alpha k, \alpha k) = \frac{1}{k \alpha^2},
\end{equation}

\begin{equation}\label{eq:Sari_12}
B_k^\alpha (p, q) = \frac{p + \alpha k (\alpha - 2)}{p + q + \alpha k (\alpha - 2)} B_k^\alpha (p - \alpha k, q).
\end{equation}

We observe that, as $\alpha \to 1$, $B_k^\alpha (x, y) \to B_k(x, y)$, and as $(k, \alpha) \to (1, 1)$, $B_k^\alpha (x, y) \to B(x, y)$.

\section{Existence and uniqueness of evolution operators}
\label{sec:evolution}
Since Z be Hilbert space with norm $\lVert \cdot\rVert$. For \( t_0 \le \frac{t^\alpha}{\alpha} \le t_f \), the spectrum of \( A(t) \) is contained in a sectorial open domain
\[ \sigma(A(t)) \subset \Sigma_\omega = \{ \lambda_\alpha \in \mathbb{C}: \, |\arg \lambda_\alpha | < \omega \}, \quad \, t_0 \le \frac{t^\alpha}{\alpha} \le t_f, \]
with some fixed angle \( 0 < \omega < \frac{\pi}{2} \). And set some of the following constraints for the family \( \{ A(t) : t_0 \leq \frac{t^\alpha}{\alpha} \leq t_f \} \) of linear operators:

\begin{enumerate}
    \renewcommand{\theenumi}{\alph{enumi})}
    \item The domain \( D(A(t)) \) of \( \{ A(t) : t_0 \leq \frac{t^\alpha}{\alpha} \leq t_f \} \) is dense in \( Z \) and independent of \( t \), and \( A(t) \) represents a closed linear operator \label{app: a}.
    \item For every $\frac{t^\alpha}{\alpha} \in [t_0, t_f]$, the resolvent $R_\alpha(\lambda_\alpha, A(t))$ is well defined for all $\lambda_\alpha$ with $\text{Re} \, \lambda_\alpha \leq 0$, and there is a positive constant $K_\alpha$ ensuring that $\| R_\alpha(\lambda_\alpha, A(t)) \| \leq \frac{K_\alpha}{|\lambda_\alpha| + 1}$ \label{app: b}.

    \item There exist constants \( 0 < \delta \leq 1 \) and \( K > 0 \) such that \( \| (A(t) - A(s)) A^{-1} (\tau) \| \leq K|t^\alpha - s^\alpha|^\delta \) for any \( \frac{t^\alpha}{\alpha}, \frac{s^\alpha}{\alpha}, \frac{\tau^\alpha}{\alpha} \in [t_0, t_f] .\)\label{app: c}

    \item For each \( \frac{t^\alpha}{\alpha} \in [t_0, t_f] \) and some \( \lambda_\alpha \in \rho(A(t)) \), the resolvent set of \( A(t) \), denoted by \( R_\alpha(\lambda_\alpha, A(t)) \), is a compact operator \label{app: d}.
\end{enumerate}

We begin with a formal computation that will lead to the construction method of the evolution operators. Suppose that for each \( \frac{t^\alpha}{\alpha} \in [t_0, t_f] \), \(-A(t)\) is the infinitesimal generator of a \( C_0 \), \(\alpha\)-conformable semigroup \( S_t(s) \), \( s \geq 0 \), in the Hilbert space \( Z \), and is defined by

$$
S_s(t-s)=\exp{\left(-(\frac{t^\alpha}{\alpha}-\frac{s^\alpha}{\alpha})A(s)\right)} = \frac{1}{2\pi i} \int_{\Gamma_\omega} e^{-\lambda_\alpha(\frac{t^\alpha}{\alpha}- \frac{s^\alpha}{\alpha})} R_\alpha(\lambda_\alpha, A(s)) \, d\lambda_\alpha,
$$
where $\Gamma_\omega : \lambda = \rho e^{\pm i \omega}, \; 0 \leq \rho < \infty$, is an integral contour.\\
Set
\begin{equation}\label{eq:Q}
    \Psi_\alpha(t,s)=S_s(t-s)+\int_\frac{s^\alpha}{\alpha}^\frac{t^\alpha}{\alpha} S_\tau (t-\tau) R(\tau ,s) d(\tau,\alpha).
    \tag{Q*}
\end{equation}
Then(formally)\\
$$ T_\alpha\bigg(\Psi_\alpha(t,s)\bigg)(t)= -A(s)S_s(t-s)+R(t,s)-\int_\frac{s^\alpha}{\alpha}^\frac{t^\alpha}{\alpha} A(\tau)S_\tau (t-\tau) R(\tau ,s) d(\tau,\alpha), $$
and\\
\begin{equation}\label{eq:Q1}
   T_\alpha\bigg(\Psi_\alpha(t,s)\bigg)(t) + A(t)\Psi_\alpha(t, s) = R(t, s) - R_1(t, s) - \int_\frac{s^\alpha}{\alpha}^\frac{t^\alpha}{\alpha} R_1(t, \tau)R(\tau, s) \, d(\tau,\alpha) ,
\end{equation}
where,
\begin{equation}\label{eq:Q2}
    R_1(t, s) = \bigg(A(s) - A(t)\bigg)S_s(t - s).
\end{equation}
Suppose, \(\Psi_\alpha(t, s)\) is an evolution operators (that is, the solution of the homogeneous part of the equation (\ref{eq:P})), it follows from (\ref{eq:Q1}) that the integral equation
\begin{equation}\label{eq:Q3}
    R(t, s) = R_1(t, s) + \int_\frac{s^\alpha}{\alpha}^\frac{t^\alpha}{\alpha} R_1(t, \tau)R(\tau, s) \, d(\tau,\alpha),
\end{equation}
must be satisfied. The main result of this section concerns the existence and uniqueness of the evolution operators for the system (\ref{eq:P}):
\begin{theorem}\label{thm:Qth1}
Under assumptions (\ref{app: a}- (\ref{app: c}, there exists a unique evolution operator for the homogeneous part of the semi-linear system (\ref{eq:P}), denoted by \( \Psi_\alpha(t, s) \), defined on the interval \( t_0 \leq \frac{s^\alpha}{\alpha} \leq \frac{t^\alpha}{\alpha} \leq t_f \), satisfying:

\begin{itemize}
\item[\( (E_1). \)] \( \| \Psi_\alpha(t, s) \| \leq M \quad \text{for} \quad t_0 \leq \frac{s^\alpha}{\alpha} \leq \frac{t^\alpha}{\alpha} \leq t_f. \)\label{app: E1}
\item[\( (E_2). \)] For \( t_0 \leq \frac{s^\alpha}{\alpha} < \frac{t^\alpha}{\alpha} \leq t_f \), \( \Psi_\alpha(t, s) : Z \to   D(A(t)) \) and \(  \frac{t^\alpha}{\alpha} \to \Psi_\alpha(t, s) \) is strongly $\alpha-$conformable differentiable in Hilbert space Z. The derivative \( T_\alpha\bigg(\Psi_\alpha(t,s)\bigg)(t)\) is a bounded operator, and it is strongly continuous on \( t_0 \leq \frac{s^\alpha}{\alpha} < \frac{t^\alpha}{\alpha} \leq t_f \). Moreover,
\begin{equation}\label{eq:Q4}
T_\alpha\bigg(\Psi_\alpha(t,s)\bigg)(t) + A(t)\Psi_\alpha(t, s) = 0, \quad \text{for} \quad t_0 \leq \frac{s^\alpha}{\alpha} < \frac{t^\alpha}{\alpha} \leq t_f, 
\end{equation}
and
\begin{equation}\label{eq:Q5}
\left\|T_\alpha\bigg(\Psi_\alpha(t,s)\bigg)(t) \right\| = \| A(t)\Psi_\alpha(t, s) \| \leq \frac{M}{t^\alpha - s^\alpha}.
\end{equation}\label{app: E2}

\item[\( (E_3). \)]  For every  $v \in  D(A(t))$  and $\frac{t^\alpha}{\alpha} \in ]t_0, t_f]$, $ \Psi_\alpha(t, s)v$ is $\alpha-$ conformable differentiable with respect to $s$ on,  \( t_0 \leq  \frac{s^\alpha}{\alpha}  \leq \frac{t^\alpha}{\alpha} \leq t_f \)  and
\begin{equation}\label{eq:Q7*}
    T_\alpha\bigg(\Psi_\alpha(t,s)\bigg)(s)v = \Psi_\alpha(t, s)A(s)v.
\end{equation}
\label{app: E3}
\end{itemize}
\end{theorem}
The proof of Theorem \ref{thm:Qth1} will occupy the majority of this section and is divided into three main parts. First, we construct \( \Psi_\alpha(t, s) \) by solving the integral equation (\ref{eq:Q3}) and defining \( \Psi_\alpha(t, s) \) using (\ref{eq:Q}). Second, we prove that \( \Psi_\alpha(t, s) \) satisfies the properties specified in $(E_2)$. Finally, we establish the uniqueness of \( \Psi_\alpha(t, s) \) and demonstrate the relation \( \Psi_\alpha(t, s) = \Psi_\alpha(t, r)\Psi_\alpha(r, s) \) for \( t_0 \leq \frac{s^\alpha}{\alpha} \leq \frac{r^\alpha}{\alpha} \leq \frac{t^\alpha}{\alpha} \leq t_f \), while also proving $(E_3)$.

Before presenting the proof, we derive some immediate consequences of assumptions ( \ref{app: a}- (\ref{app: c}. In particular, assumption (\ref{app: b} and the denseness of $D(A(t))$ in $Z$ imply that for every $\frac{t^\alpha}{\alpha} \in [t_0, t_f]$, the operator $-A(t)$ serves as the infinitesimal generator of an analytic semigroup $S_t(s)$, $s \geq 0$, which satisfies the following bound (see \cite{bib34}):
\begin{equation}\label{eq:S}
    \|S_t(s)\| \leq M, \quad \text{for } s \geq 0.
\end{equation}

and
\begin{equation}\label{eq:A(t)S(t)}
    \|A(t)S_t(s)\| \leq \frac{M}{s^\alpha}, \quad \text{for } s > 0.
\end{equation}

\begin{lemma}\label{lemma:Q1}
    Let \( (\ref{app: a} - (\ref{app: c} \) be satisfied, then
     \begin{equation}\label{eq:Q7}
\| (A(t_3) - A(t_4))S_\tau(s) \| \leq \frac{M}{s^\alpha} |t_3^\alpha - t_4^\alpha|^\delta \quad \text{for }  \frac{s^\alpha}{\alpha} \in ]t_0, t_f], \, \frac{t_3^\alpha}{\alpha}, \frac{t_4^\alpha}{\alpha} \in [t_0, t_f].
\end{equation}

Moreover, \( A(t)S_\tau(s) \in BL(Z) \) for \( \frac{s^\alpha}{\alpha} \in ]t_0, t_f] \), \( \frac{\tau^\alpha}{\alpha} \in [t_0, t_f] \), and \( \frac{t^\alpha}{\alpha} \in [t_0, t_f] \). Furthermore, the \( BL(Z) \)-valued function \( A(t)S_\tau(s) \) is uniformly continuous in the uniform operator topology for \( \frac{s^\alpha}{\alpha} \in [\epsilon, t_f] \), \( \frac{\tau^\alpha}{\alpha} \in [t_0, t_f] \), and \( \frac{t^\alpha}{\alpha} \in [t_0, t_f] \) for every \( \epsilon > 0 \).

\end{lemma}
\begin{proof}
    Since \( S_{\tau}(s): Z \to D(A(t)) \) for \( s > 0 \), and from (\ref{eq:A(t)S(t)}), it follows by the closed graph theorem that \( A(t) S_\tau(s) \) is a bounded linear operator for \( \frac{t^\alpha}{\alpha}, \frac{\tau^\alpha}{\alpha} \in [t_0, t_f] \) and \( \frac{s^\alpha}{\alpha} \in ]t_0, t_f] \). Moreover, using (\ref{eq:A(t)S(t)}) and assumption (\ref{app: c}, we have

     \begin{align*}
         \lVert (A(t_3)-A(t_4))S_\tau (s) \rVert &\leq \lVert (A(t_3)-A(t_4)) A(\tau)^{-1} \rVert \quad \lVert A(\tau) S_\tau (s)\rVert\\
         & \leq \frac{M}{s^\alpha}|t_4 -t_3|^\delta.
     \end{align*}
\end{proof}
\begin{corollary}\label{corollary:Q1}
   The operator \( R_1(t, s) \), as defined in equation (\ref{eq:Q2}), is uniformly continuous in the uniform operator topology for \( t_0 \leq \frac{s^\alpha}{\alpha} \leq \frac{t^\alpha}{\alpha} - \epsilon \leq t_f \), for every \( \epsilon > 0 \) and

\begin{equation}\label{eq:Q10}
\| R_1(t, s) \| \leq M (t^\alpha - s^\alpha)^{\delta - 1} , \quad \text{for } \quad t_0 \leq  \frac{s^\alpha}{\alpha} <  \frac{t^\alpha}{\alpha} \leq t_f.
\end{equation}
\end{corollary}
\begin{proof}
 The first part of the claim directly follows from the uniform continuity of \( A(t)S_t(s) \) in \( BL(Z) \), while equation (\ref{eq:Q10}) follows from the estimate
 \begin{align*}
    \| R_1(t, s) \| &\leq \| (A(t) - A(s))A(s)^{-1} \| \| A(s)S_s(t - s) \| \\
    & \leq M |t^\alpha - s^\alpha|^{\delta} |t^\alpha - s^\alpha|^{-1}\\
    &  = M |t^\alpha - s^\alpha|^{\delta - 1}.
 \end{align*}

\end{proof}

The proof of the first part of Theorem \ref{thm:Qth1} is provided below:\\

By Corollary \ref{corollary:Q1} and  the integral equation (\ref{eq:Q3}) can be solved by successive approximation:
\begin{equation}\label{eq:Q11}
R(t, s) = \sum_{m=1}^{\infty} R_m(t, s), 
\end{equation}
\begin{equation}\label{eq:Q12}
R_m(t, s) = \int_{\frac{s^\alpha}{\alpha}}^{\frac{t^\alpha}{\alpha}} R_{1}(t, \sigma)R_{m-1}(\sigma,s)  d(\sigma,\alpha), \quad m = 2, 3, \dots 
\end{equation}

 \begin{lemma}\label{lemma:Q2}
The operator \( R(t, s) \), as defined in (\ref{eq:Q3}), is continuous in \( t_0 \leq \frac{s^\alpha}{\alpha} < \frac{t^\alpha}{\alpha} \leq t_f \) in the uniform operator topology and satisfies

\begin{equation}\label{eq:Q13}
\|R(t, s)\| \leq M|t^\alpha - s^\alpha|^{\delta-1}.
\end{equation}
 \end{lemma}
       
    \begin{proof}
       This follows from
\begin{equation}\label{eq:Q14}
\| R_m(t, s) \| \leq\frac{M^m}{\alpha^{m-1}}  \frac{(\Gamma(\delta))^{m-1} \Gamma(\delta+\alpha-1)}{\Gamma(m\delta+\alpha-1)} (t^\alpha - s^\alpha)^{m \delta - 1},
\end{equation}
where $M$ is the same constant as in  (\ref{eq:Q10}) and from the preceding corollary \ref{corollary:Q1}.
   \end{proof}

\(\Psi_\alpha(t, s)\) defined by (\ref{eq:Q}) follows readily from the strong continuity of \(S_s(\tau)\), (\ref{eq:S}), and (\ref{eq:Q13}) that \(\Psi_\alpha(t, s)\) is strongly continuous for \(t_0 \leq \frac{s^\alpha}{\alpha} < \frac{t^\alpha}{\alpha} \leq t_f\), and that

\begin{equation}\label{Eq:equationQ1}
    \|\Psi_\alpha(t, s)\| \leq M.
\end{equation}

The proof of the second part of Theorem \ref{thm:Qth1} is provided below:\\
For sufficiently small positive $h$, we define

\begin{equation}\label{eq:Q15}
\Psi_{\alpha,h}(t, s) = \exp \left( - (\frac{t^\alpha}{\alpha}-\frac{s^\alpha}{\alpha})A(s) \right) + \int_ \frac{s^\alpha}{\alpha}^{ \frac{t^\alpha}{\alpha}-\frac{h^\alpha}{\alpha}}\exp \left( - (\frac{t^\alpha}{\alpha}-\frac{\tau^\alpha}{\alpha})A(\tau) \right) R(\tau, s) d(\tau,\alpha).
\end{equation}

Then we have,
$$
T_\alpha \bigg(\Psi_{\alpha,h}(t, s)\bigg)(t) + A(t)\Psi_{\alpha,h}(t, s) $$ $$
= -R_1(t, s) + \exp \left( - \frac{h^\alpha}{\alpha} A(s) \right) R(t - h, s) - \int_{ \frac{s^\alpha}{\alpha}}^{ \frac{t^\alpha}{\alpha}-\frac{h^\alpha}{\alpha}} R_1(t, \tau)R(\tau, s) d(\tau,\alpha).
$$

The  right side member is uniformly bound in $h$ and then
\begin{equation}\label{eq:Q16}
T_\alpha\bigg(\Psi_{\alpha,h}(t, s)\bigg)(t) + A(t)\Psi_{\alpha,h}(t, s) \to 0, \quad \text{strongly as } h \to 0,
\end{equation}

That is, 
$$ T_\alpha \bigg(\Psi_\alpha(t, s)\bigg)(t) + A(t)\Psi_\alpha(t, s) = 0, \quad \text{for} \quad t > s. $$

This concludes the proofs of (\ref{eq:Q4}) and (\ref{eq:Q5}).

The proof of the third part of Theorem \ref{thm:Qth1} is provided below:\\
To prove it  $(E_3)$, we construct an operator-valued function \( \Phi_\alpha(t, s) \) that satisfies

\begin{equation}\label{eq:Q17}
    \begin{cases}
T_\alpha \bigg(\Phi_\alpha(t, s)\bigg)(s) v = \Phi_\alpha(t, s) A(s) v, & \text{for } t_0 \leq \frac{s^\alpha}{\alpha} \leq \frac{t^\alpha}{\alpha} \leq t_f, \, v \in   D(A(t)),\\
\Phi_\alpha(t, t) = I,
\end{cases}
\end{equation}
Later, we will prove that \( \Phi_\alpha(t, s) = \Psi_\alpha(t, s) \). The construction of \( \Phi_\alpha(t, s) \) follows the same procedure as the construction of \( \Psi_\alpha(t, s) \) outlined above. We set

\begin{align*}
    Q_1(t, s)& = T_\alpha \bigg(S_s(t - s)\bigg)(t)+T_\alpha \bigg(S_s(t - s)\bigg)(s) \\
    &= \frac{1}{2 \pi i} \int_{\Gamma_w} e^{-\lambda_\alpha (\frac{t^\alpha}{\alpha} - \frac{s^\alpha}{\alpha})} T_\alpha \bigg( R_\alpha(\lambda_\alpha : A(s))\bigg)(s) \, d\lambda_\alpha.
\end{align*}

Using assumption (\ref{app: b} and the estimation of  $Q_1(t,s)$ , we find

\[
\| Q_1(t, s) \| = \left\| \frac{1}{2 \pi i} \int_{\Gamma_w} e^{-\lambda_\alpha (\frac{t^\alpha}{\alpha} - \frac{s^\alpha}{\alpha})} \quad T_\alpha \bigg( R_\alpha(\lambda_\alpha : A(s))\bigg) (s)\, d\lambda_\alpha \right\| \leq M.
\]
Next, we solve the integral equation  by successive approximations.
\begin{equation}\label{eq:Q18}
    Q(t, s) = Q_1(t, s) + \int_ \frac{s^\alpha}{\alpha}^ \frac{t^\alpha}{\alpha} Q(t, \tau) Q_1(\tau, s) \, d(\tau,\alpha).
\end{equation}

This is carried out in the same manner as solving the integral equation (\ref{eq:Q3}). Given that \( Q_1(t, s) \) is uniformly bounded in this case, the solution \( Q(t, s) \) of (\ref{eq:Q18}) will satisfy this condition.
\[
\| Q(t, s) \| \leq M.
\]
Setting
\begin{equation}\label{eq:phi_alpha}
    \Phi_\alpha(t, s) = S_s(t - s) + \int_ \frac{s^\alpha}{\alpha}^ \frac{t^\alpha}{\alpha} Q(t, \tau) S_s(\tau - s) \, d(\tau,\alpha).
\end{equation}

We have to show that \( \| \Phi_\alpha(t, s) \| \leq M \), and for \( v \in  D(A(t)) \), \( \Phi_\alpha(t, s) v \) is \(\alpha\)-conformably differentiable with respect to \( s \). Differentiating \( \Phi_\alpha(t, s) v \) with respect to \( s \) in the \(\alpha\)-conformable sense yields
\begin{align*}
   & T_\alpha \bigg(\Phi_\alpha(t, s)\bigg)(s) v - \Phi_\alpha(t, s) A(s) v\\
& = Q_1(t, s) v + \int_{\frac{s^\alpha}{\alpha}}^{\frac{t^\alpha}{\alpha}} Q(t, \tau) Q_1(\tau, s) v \, d(\tau,\alpha) - Q(t, s) v\\
\end{align*}
From (\ref{eq:Q18}),
\begin{align*}
     T_\alpha \bigg(\Phi_\alpha(t, s)\bigg)(s) v - \Phi_\alpha(t, s) A(s) v
& = 0.
\end{align*}
By the definition of \( \Phi_\alpha(t, s) \), it follows that \( \Phi_\alpha(t, t) = I \), and hence \( \Phi_\alpha(t, s) \) is a solution to (\ref{eq:Q17}).

Consider \( z_0 \in Z \) and \( \frac{s^\alpha}{\alpha} < \frac{r^\alpha}{\alpha} < \frac{t^\alpha}{\alpha} \), the function \( \frac{r^\alpha}{\alpha} \to \Phi_\alpha(t, r)\Psi_\alpha(r, s)z_0 \) is \( \alpha \)-conformable differentiable in \( r \). Moreover, we have
\begin{align*}
    T_\alpha \bigg(\Phi_\alpha(t, r)\Psi_\alpha(r, s)\bigg)(r)z_0 & = \Phi_\alpha(t, r)A(r)\Psi_\alpha(r, s)z_0 - \Phi_\alpha(t, r)A(r)\Psi_\alpha(r, s)z_0 \\
    &= 0.
\end{align*}

This indicates that \( \Phi_\alpha(t, r)\Psi_\alpha(r, s)z_0 \) is independent of \( r^\alpha \) for \( s^\alpha < r^\alpha < t^\alpha \). By taking the limits \( r^\alpha \downarrow s^\alpha \) and \( r^\alpha \uparrow t^\alpha \), we obtain \( \Phi_\alpha(t, s)z_0 = \Psi_\alpha(t, s)z_0 \) for every \( z_0 \in Z \). Consequently, \( \Psi_\alpha(t, s) = \Phi_\alpha(t, s) \), and \( \Psi_\alpha(t, s) \) satisfies 

\begin{equation}\label{eq:Q19}
    T_\alpha \big(\Psi_\alpha(t, s)\big)(s) v = \Psi_\alpha(t, s) A(s) v, \quad \text{for } v \in D(A(t)),
\end{equation}

as required.

To complete the proof of Theorem \ref{thm:Qth1}, it remains to demonstrate the uniqueness of $\Psi_\alpha(t,s)$ and to verify that it satisfies the property $\Psi_\alpha(t,s) = \Psi_\alpha(t,r)\Psi_\alpha(r,s)$ for $t_0 \leq \frac{s^\alpha}{\alpha} \leq \frac{r^\alpha}{\alpha} \leq \frac{t^\alpha}{\alpha} \leq t_f.$

Both these claims follow from: 

\begin{theorem}\label{thm:thQ2}

    Let \(A(t)\), with \(t_0 \leq \frac{t^\alpha}{\alpha} \leq t_f\), satisfy the conditions \((\ref{app: a}-(\ref{app: c}\). For every \(t_0 \leq \zeta < \frac{t^\alpha}{\alpha} \leq t_f\) and \(z_0 \in Z\), the initial value problem
\begin{equation}\label{eq:Q20}
\begin{cases}
(T_\alpha u)(t) + A(t)u(t) = 0,  \quad \alpha \in (0,1] ,   \\
u(\zeta) = z_0, \quad  t \in [t_1,t_2], \quad \text{and} \quad \zeta \in [t_0,t_f] \cap [t_1,t_2],
\end{cases}
\end{equation}
has a unique solution \(u(t)\) given by \(u(t) = \Psi_\alpha(t,\zeta)z_0\), where \(\Psi_\alpha(t,s)\) is the evolution operator constructed above.
\end{theorem}
\begin{proof}
   From (\ref{eq:Q4}), it shows that \(u(t) = \Psi_\alpha(t,s)z_0\) is a solution of homogeneous system (\ref{eq:Q20}). To prove its uniqueness, suppose that \(v(t)\) is another solution of a homogeneous system (\ref{eq:Q20}). Since \(v(r) \in D(A(r))\) for every \(r > s\), it follows from (\ref{eq:Q7*}) and (\ref{eq:Q20}) that the function \(r \mapsto \Psi_\alpha(t,r)v(r)\) is $\alpha-$  conformable differentiable with respect to $r$, and  
\[
T_\alpha \bigg(\Psi_\alpha(t,r)v(r)\bigg)(r) = \Psi_\alpha(t,r)A(r)v(r) - \Psi_\alpha(t,r)A(r)v(r) = 0.
\]
It implies that, \(\Psi_\alpha(t,r)v(r)\) is constant on \(s^\alpha < r^\alpha < t^\alpha\). Since this is continuous on \(s^\alpha \leq r^\alpha \leq t^\alpha\), so we can take limit \(r^\alpha \to t^\alpha\) and \(r^\alpha \to s^\alpha\), to obtain \(\Psi_\alpha(t,s)z_0 = v(t)\). The uniqueness of the solution of (\ref{eq:Q20}) then follows.  
\end{proof}

From Theorem \ref{thm:thQ2}, it shows that for \(z_0 \in Z\),  
\[
\Psi_\alpha(t,s)z_0 = \Psi_\alpha(t,r)\Psi_\alpha(r,s)z_0 \quad \text{for} \quad t_0 \leq \frac{s^\alpha}{\alpha} \leq \frac{t^\alpha}{\alpha} \leq t_f,
\]
and \(\Psi_\alpha(t,s)\) is thus an evolution operator satisfying  $(E_1)$, $(E_2)$, and $(E_3)$. If \(\Phi_\alpha(t,s)\) is an evolution operator that satisfies  $(E_1)$ and $(E_2)$, then \(\Phi_\alpha(t,s)z_0\) is a solution of (\ref{eq:Q20}). Furthermore, from Theorem \ref{thm:thQ2}, it follows that  
\[
\Phi_\alpha(t,s)z_0 = \Psi_\alpha(t,s)z_0 \implies \Phi_\alpha(t,s) = \Psi_\alpha(t,s),
\]
which is the unique evolution operator satisfying $(E_1)$, $(E_2)$, and  $(E_3)$. This concludes the proof of Theorem \ref{thm:Qth1}.

\begin{remark}
   From the above theorems \ref{thm:Qth1} and \ref{thm:thQ2}, it is clear that the evolution operator $\Psi_\alpha(t,s)$ holds for any $0 < \alpha < \infty$.
\end{remark}

\begin{lemma}\label{lemma:Q3}
    The family of operators \(\{\Psi_\alpha(t, s) :  \frac{t^\alpha}{\alpha} > \frac{s^\alpha}{\alpha} \geq t_0\}\) is continuous in \(t\), in the uniform operator topology uniformly for \(s\).
\end{lemma}
In exploring the exact null controllability of the semi-linear system (\ref{eq:P}), we examine the associated linear system:
\[
\begin{cases}
(T_{\alpha}z)(t) + A(t)z(t) = Bu(t) + h(t), & t \in [t_1,t_2], \quad \alpha \in (0,1], \\
z(\zeta) = z_0, & \zeta \in [t_0,t_f] \cap [t_1,t_2].
\end{cases}
\label{eq:Z}
\tag{R}
\]

This system is associated with equation (\ref{eq:P}), where there must exist constants \( 0 \leq \xi < \min\{\alpha, \delta\} \) such that \( h \in L^{\frac{1}{\xi}}([\zeta, t_2]; Z) \).\\

We define the operators $L^{t_2}_{\zeta}: L^2([\zeta,t_2];U) \rightarrow Z$ and $N^{t_2}_{\zeta}: Z \times L^{\frac{1}{\xi}}([\zeta,t_2]; Z) \rightarrow Z$, respectively, as follows:
\begin{align*}
   & L^{t_2}_{\zeta} u = \int_{\zeta}^{t_2} \Psi_\alpha(t_2, s)Bu(s) \,d(s,\alpha), \quad \text{for } u \in L^2([\zeta,t_2];U). \\
    & N^{t_2}_{\zeta}(z_0, h)= \Psi_\alpha(t_2, \zeta)z_0 + \int_{\zeta}^{t_2} \Psi_\alpha(t_2, s)h(s) \, d(s,\alpha),\quad \textit{for }  (z_0, h) \in Z \times L^{\frac{1}{\xi}}([\zeta,t_2]; Z).
\end{align*}
Then, we have the following definition:
\begin{definition}
   The linear system (\ref{eq:Z}) is said to be exactly null controllable on $[\zeta,t_2]$ if the following condition holds:
\[
\text{Im } L^{t_2}_{\zeta} \supseteq \text{Im } N^{t_2}_{\zeta}.
\]
\end{definition}

\begin{remark}\label{remark1}
    As demonstrated in \cite{bib35}, the system (\ref{eq:Z}) is exactly null controllable if and only if there exists a positive constant $\gamma$ such that.
\[
\| (L^{t_2}_{\zeta})^* z \| \geq \gamma \| (N^{t_2}_{\zeta})^* z \|, \quad \text{for all } z \in Z.
\]
\end{remark}

The following result is particularly significant for our discussion.

\begin{lemma}\label{lemma:2}
    Assuming that the linear system (\ref{eq:Z}) is exactly null controllable, then the linear operator $H_\alpha := (L_{\zeta})^{-1}(N_{\zeta}^{t_2}): Z \times L^{\frac{1}{\xi}}([\zeta,t_2]; Z) \to L^2([\zeta,t_2]; U)$ is bounded, and the control
  \begin{align}\label{eq:u(t)}
u(t) &:= -\bigg( (L_{\zeta})^{-1} N_{\zeta}^{t_2}(z, h)\bigg)(t) \notag \\
&= -H_\alpha (z_0, h)(t) \notag \\
&= - (L_{\zeta})^{-1} \left( \Psi_\alpha(t_2, \zeta) z_0 
+ \int_{\zeta}^{t_2} \Psi_\alpha(t_2, s) h(s) \, d(s, \alpha) \right)(t),
\end{align}

linear system  (\ref{eq:Z}) steers from $z_0$ to $0$, where $L_{\zeta}$ is the restriction of $L_{\zeta}^{t_2}$ to $[\text{ker } L_{\zeta}^{t_2}]^\perp$.
\end{lemma}
\begin{proof} 

Since \( L_\zeta^{t_2} \) is a bounded linear operator, but not necessarily one-to-one. Define the null space of \( L_\zeta^{t_2} \) as  
\[
\ker  L_\zeta^{t_2} = \{ u \in L^2([\zeta,t_2]; U) :  L_\zeta^{t_2} u = 0 \}.
\]  
Let \( [\ker  L_\zeta^{t_2}]^\perp \) denote the orthogonal complement of \( \ker  L_\zeta^{t_2} \) in \( L^2([\zeta,t_2]; U) \). Consider the restriction of \(  L_\zeta^{t_2} \) to \( [\ker  L_\zeta^{t_2}]^\perp \), defined as  
\[
 L_\zeta : [\ker  L_\zeta^{t_2}]^\perp \to \mathrm{Im}( L_\zeta^{t_2}).
\]  
The operator \(  L_\zeta \) is necessarily one-to-one. Next, define  
\[
H_\alpha  : Z \times L^\frac{1}{\zeta}([\zeta,t_2]; Z) \to L^2([\zeta,t_2]; U), \quad H_\alpha (z, h) = ( L_\zeta)^{-1}N_{\zeta}^{t_2}(z, h).
\]

We will now show that \( H_\alpha  \) is a bounded linear operator. By the inverse mapping theorem, \( ( L_\zeta)^{-1} \) is bounded if both \( [\ker  L_\zeta^{t_2}]^\perp \) and \( \mathrm{Im}( L_\zeta^{t_2}) \) are Banach spaces. While \( [\ker  L_\zeta^{t_2}]^\perp \) is closed, \( \mathrm{Im}( L_\zeta^{t_2}) \) may not necessarily be closed. To establish that \( H_\alpha  \) is bounded, consider the following argument: Let \( (z_n, h_n) \) be a convergent sequence in \( Z \times L^\frac{1}{\zeta}([\zeta,t_2]; Z) \) such that \( H_\alpha (z_n, h_n) \) converges in \( Z \). Define  
\[
(z, h) = \lim_{n \to \infty} (z_n, h_n) \quad \text{and} \quad -u = \lim_{n \to \infty} H_\alpha (z_n, h_n).
\]  
Since \( [\ker  L_\zeta^{t_2}]^\perp \) is closed, it follows that \( u \in [\ker  L_\zeta^{t_2}]^\perp \). Moreover, because \(  L_\zeta^{t_2} \) and \( N_{\zeta}^{t_2} \) are continuous, we have  
\[
 L_\zeta^{t_2} u + N_{\zeta}^{t_2}(z, h) = \lim_{n \to \infty} \left( - L_\zeta^{t_2} H_\alpha (z_n, h_n) + N_{\zeta}^{t_2}(z_n, h_n) \right) = 0.
\]

By the construction of \(  L_\zeta \), it follows that \( u = -( L_\zeta)^{-1} N_{\zeta}^{t_2}(z, h) = -H_\alpha (z, h) \). Hence, \( H_\alpha  \) is a closed operator. By the closed graph theorem, \( H_\alpha  \) is bounded.

\end{proof}
\section{Exact null controllability of non-autonomous systems}
\label{sec:Exact}
This section delves into the exact null controllability problem concerning semi-linear systems (\ref{eq:P}). Initially, we define a mild solution and its exact null controllability.
\begin{definition}
  We define $x \in \mathbb{C}([\zeta,t_2]; Z)$ as a mild solution to problem (\ref{eq:P}) if it satisfies the following equation:
\begin{equation}
x(t) = \Psi_\alpha(t,\zeta)x_{0} + \int_{\zeta}^{t} \Psi_\alpha(t, s)Bu(s) \frac{ds}{s^{1-\alpha}} + \int_{\zeta}^{t} \Psi_\alpha(t,s)F(s,x(s)) \frac{ds}{s^{1-\alpha}},
\label{eq:P7}
\end{equation}
where $u(\cdot) \in L^2([\zeta,t_2];U)$ is the control function.
\end{definition}
\begin{definition}
    A system (\ref{eq:P}) is considered exactly null controllable if there exists a $u \in L^2([\zeta,t_2]; U)$ such that under this control, $x(t_2, u) = 0$.
\end{definition}
 We enforce the following constraints on the system (\ref{eq:P}):

\begin{enumerate}
    \item[($H_1$)] Let $t \in [\zeta,t_2]$, the function $F(t,\cdot):Z \rightarrow Z$ remain continuous. For any given $x \in Z$, the function $F(\cdot,x): [\zeta,t_2] \rightarrow Z$ is taken to be strongly measurable. Furthermore, there must exist constants $0\leq \xi < \min \{\alpha,\delta\},$ $\gamma>0$, For any $r>0$ and for all functions $h_r(t) \in L^{\frac{1}{\xi}}([\zeta,t_2]; \mathbb{R}^+)$
\[
\sup_{x \in W_r} \| F(t, x) \| \leq h_r(t),
\]
that hold for almost every $t \in [\zeta,t_2]$, with
\[
\liminf_{r \rightarrow +\infty} \frac{\| h_r(t) \|_{L^\frac{1}{\xi}}}{r} = \gamma < \infty,
\]
where $W_r = \{ x \in \mathbb{C}([\zeta,t_2]; Z): \| x(\cdot) \| \leq r,\quad r=r(\alpha) \},$ and this set of functions depends on \(\alpha\). \label{app: H} 
    \item[($H_2$)]\label{assum:H2_unique}  In $Z$, the linear system (\ref{eq:Z}) is exact null controllability on $[\zeta,t_2]$.\label{app: H2} 
\end{enumerate}

Next, for convenience, let us introduce the following notation:
$$\psi(r)=\sup \{h_r(\rho):\lVert \rho \rVert \leq r\},\quad \textit{then} \quad \liminf_{r \rightarrow +\infty} \frac{\| \psi(r) \|_{L^\frac{1}{\xi}}}{r} = \gamma ,$$ 
\begin{equation}\label{eq:N}
    \textit{and} \quad
N = \max \left\{ \bigg(  (t_2)^{\frac{\alpha-\xi}{1-\xi}}-(\zeta)^{\frac{\alpha-\xi}{1-\xi}} \bigg)^{1-\xi} \bigg(\frac{1-\xi}{\alpha-\xi}\bigg)^{1-\xi} :0\leq \xi < \min \{\alpha,\delta\}  \right\}.
\end{equation}

\begin{theorem}\label{thm:th6}
Let $x_0 \in Z$ and $W_r = \{ x \in \mathbb{C}([\zeta, t_2]; Z) : \| x(\cdot) \| \leq r \}, r > 0.$ If $-A(t)$ is the generator of the implicit evolution operator $\Psi_{\alpha}(t, s)$, and the nonlinear function satisfies Assumption ($H_1$), then the nonlinear system (\ref{eq:P}) has a mild solution on $[\zeta, t_2]$, provided the following conditions:  
\begin{equation}
M N\gamma \bigg( \lVert B \rVert \lVert H_\alpha  \rVert + 1 \bigg) < 1,  
\label{eq:P8}
\end{equation}
and there exists an $r$ such that  
\begin{equation}
\lVert H_\alpha  \rVert \bigg( \lVert x_0 \rVert + \psi(r) \sqrt{t_2 - \zeta} \bigg) \leq r.  
\end{equation}

\end{theorem}

\begin{proof}

 The proof will be given in several steps:

\textbf{Step 1:} The control function defined by (\ref{eq:u(t)}), $u(\cdot) = -H_\alpha (x_0, F)(s)$, is bounded on $W_r$. Indeed,
\begin{align*}
    \lVert u \rVert & = \bigg( \int_{\zeta}^{t_2} \lVert H_\alpha (x_0, F)(\tau) \rVert^2 \, d\tau \bigg)^\frac{1}{2} \\
    &\leq \lVert H_\alpha  \rVert\bigg( \lVert x_0\rVert +\bigg(\int_{\zeta}^{t_2} \bigg(   \lVert F(\tau,x(\tau)\rVert\, \bigg)^2 \, d\tau \bigg)^\frac{1}{2}\bigg)\\
    &\leq \lVert H_\alpha  \rVert\bigg( \lVert x_0\rVert +\bigg(\int_{\zeta}^{t_2} \bigg(   \lVert h_r(\tau)\rVert\, \bigg)^2 \,  d\tau \bigg)^\frac{1}{2}\bigg)\\
    & \leq  \lVert H_\alpha  \rVert\bigg( \lVert x_0\rVert +\bigg(\int_{\zeta}^{t_2} \bigg( \psi(r) \bigg)^2 \,  d\tau \bigg)^\frac{1}{2}\bigg)\\
    & =  \lVert H_\alpha  \rVert\bigg( \lVert x_0\rVert +\psi(r) \sqrt{t_2-\zeta}\bigg).\\
\end{align*}
\textbf{Step 2:}

Let $Q_\alpha : \mathbb{C}([\zeta,t_2]; Z) \rightarrow \mathbb{C}([\zeta,t_2]; Z)$ be defined as follows:
\[
(Q_\alpha x)(t) = \Psi_\alpha(t,\zeta)x_{0} + \int_{\zeta}^{t} \Psi_\alpha(t, s)Bu(s) \frac{ds}{s^{1-\alpha}} + \int_{\zeta}^{t} \Psi_\alpha(t,s)F(s,x(s)) \frac{ds}{s^{1-\alpha}}
\]
We prove the existence of a fixed point for $Q_\alpha$ in $\mathbb{C}([\zeta,t_2]; Z)$. To begin, we establish the existence of $r > 0$ such that $Q_\alpha$ maps $W_r$ into itself. Imagine that this is not the case. Then, for any $r > 0$ and $x(\cdot) \in W_r$, there exists some $t(r) \in [\zeta,t_2]$ such that $\| Q_\alpha x (t) \| > r$. However, considering Equation (\ref{Eq:equationQ1}) along with Assumption ($H_1$), this leads to
\begin{align*}
    r &< \lVert (Q_\alpha x)(t) \rVert \\
& \leq  \|\Psi_\alpha(t,\zeta)x_{0}\| + \left\| \int_{\zeta}^{t} \Psi_\alpha(t, s)B H_\alpha (x_0,F)(s) \frac{ds}{s^{1-\alpha}} \right\| + \left\| \int_{\zeta}^{t} \Psi_\alpha(t,s)F(s,x(s)) \frac{ds}{s^{1-\alpha}} \right\|\\
& \leq  \|\Psi_\alpha(t,\zeta)x_{0}\| +  \int_{\zeta}^{t} \left\|\Psi_\alpha(t, s)B H_\alpha (x_0,F)(s)\right\| \frac{ds}{s^{1-\alpha}}  +  \int_{\zeta}^{t}\left\| \Psi_\alpha(t,s)F(s,x(s))\right\| \frac{ds}{s^{1-\alpha}}\\
&\leq M \| x_0\|+M \| B\| \cdot \int_{\zeta}^{t_2} \| H_\alpha (x_0,F)(s)   \|\frac{ds}{s^{1-\alpha}}+ M \int_{\zeta}^{t_2} \|h_r(s)\| \frac{ds}{s^{1-\alpha}}\\
&\leq  M \| x_0\|+M \| B\| \cdot  \bigg(\int_{\zeta}^{t_2} s^\frac{{\alpha-1}}{1-\xi} \ ds \bigg)^{1-\xi} \bigg(\int_{\zeta}^{t_2} \lVert H_\alpha (x_0,F)(s)\rVert^\frac{1}{\xi} ds \bigg)^\xi\\
&+ M \bigg(\int_{\zeta}^{t_2} s^\frac{{\alpha-1}}{1-\xi} \ ds \bigg)^{1-\xi} \bigg(\int_{\zeta}^{t_2} \lVert h_r(s)\rVert^\frac{1}{\xi} ds \bigg)^\xi \\
&\leq  M \| x_0\|+M \| B\|  \bigg(\int_{\zeta}^{t_2} s^\frac{{\alpha-1}}{1-\xi} \ ds \bigg)^{1-\xi} \cdot \lVert H_\alpha \rVert\bigg( \lVert x_0\rVert +\bigg(\int_{\zeta}^{t_2} \bigg(   \lVert F(s,x(s)\rVert\,  \bigg)^\frac{1}{\xi} \, ds \bigg)^\xi \bigg)\\
&+ M \bigg(\int_{\zeta}^{t_2} s^\frac{{\alpha-1}}{1-\xi} \ ds \bigg)^{1-\xi} \bigg(\int_{\zeta}^{t_2} \lVert h_r(s)\rVert^\frac{1}{\xi} ds \bigg)^\xi \\
&\leq  M \| x_0\|+M \| B\| \cdot \lVert H_\alpha \rVert \bigg(\int_{\zeta}^{t_2} s^\frac{{\alpha-1}}{1-\xi} \ ds \bigg)^{1-\xi} \bigg( \lVert x_0\rVert +\bigg(\int_{\zeta}^{t_2}  \lVert h_r(s)\rVert^\frac{1}{\xi} \, ds \bigg)^\xi \bigg)\\
&+ M \bigg(\int_{\zeta}^{t_2} s^\frac{{\alpha-1}}{1-\xi} \ ds \bigg)^{1-\xi} \bigg(\int_{\zeta}^{t_2} \lVert h_r(s)\rVert^\frac{1}{\xi} ds \bigg)^\xi \\
&\leq  M \| x_0\|\\
&+M \| B\|  \lVert H_\alpha \rVert \bigg(  (t_2)^{\frac{\alpha-\xi}{1-\xi}}-(\zeta)^{\frac{\alpha-\xi}{1-\xi}} \bigg)^{1-\xi} \bigg(\frac{1-\xi}{\alpha-\xi}\bigg)^{1-\xi} \bigg( \lVert x_0\rVert +\bigg(  \int_{\zeta}^{t_2}  \lVert h_r(s)\rVert^\frac{1}{\xi} \, ds \bigg)^\xi \bigg)\\
&+ M \bigg(  (t_2)^{\frac{\alpha-\xi}{1-\xi}}-(\zeta)^{\frac{\alpha-\xi}{1-\xi}} \bigg)^{1-\xi} \bigg(\frac{1-\xi}{\alpha-\xi}\bigg)^{1-\xi} \bigg(\int_{\zeta}^{t_2} \lVert h_r(s)\rVert^\frac{1}{\xi} ds \bigg)^\xi \\
&\leq  M \| x_0\|+ M \| B\|  \lVert H_\alpha \rVert N \bigg( \lVert x_0\rVert +\lVert h_r(s)\rVert_{L^\frac{1}{\xi}} \bigg) + M N  \lVert h_r(s)\rVert_{L^\frac{1}{\xi}}
\end{align*}             
That is,
\begin{equation}\label{eq:r}
    r\leq  M \| x_0\|+M \| B\| \cdot \lVert H_\alpha \rVert N \bigg( \lVert x_0\rVert +\lVert h_r(s)\rVert_{L^\frac{1}{\xi}} \bigg) + M N  \lVert h_r(s)\rVert_{L^\frac{1}{\xi}}
\end{equation}

Dividing both sides of equation (\ref{eq:r}) by \( r \) and considering the limit as \( r \to +\infty \), we obtain:

$$M \lVert B\rVert \lVert H_\alpha \rVert N \gamma +M N \gamma \geq 1,$$

The statement contradicts equation (\ref{eq:P8}). Therefore, there exists a positive number \( r > 0 \) such that \( Q_\alpha (W_r) \subseteq W_r \).\\

Next, we establish the complete continuity of $Q_\alpha$ through a three-step process.\\

\textbf{Step 3:} We have to show that $Q_\alpha$ is continuous. Let ${x}_k \rightarrow {x}$ in ${W}_r$. Then,

\begin{align*}
    \| (Q_\alpha x_k)(t)-(Q_\alpha x)(t)\| 
    &= \|\Psi_\alpha(t,\zeta)x_{0}-\Psi_\alpha(t,\zeta)x_{0}\|\\
    &+ \left\| \int_{\zeta}^{t} \Psi_\alpha(t, s)B \bigg[H_\alpha (x_0,F(s,x_k(s)))-H_\alpha (x_0,F(s,x(s))\bigg] \frac{ds}{s^{1-\alpha}} \right\|\\
    &+ \left\| \int_{\zeta}^{t} \Psi_\alpha(t,s)\bigg[F(s,x_k(s))-F(s,x(s))\bigg] \frac{ds}{s^{1-\alpha}} \right\|.
\end{align*}

Given that \( H_\alpha  \) is a continuous function in \( \mathbb{C}([\zeta,t_2]; Z) \), and \( F(t, \cdot): Z \rightarrow Z \) is also continuous, the Lebesgue Dominated Convergence Theorem immediately implies that

$$\|(Q_\alpha x_k)(t) - (Q_\alpha x)(t)\| \rightarrow 0, \text{ as } k \rightarrow +\infty,$$

that is, $Q_\alpha$ is continuous.\\

\textbf{Step 4:} We aim to demonstrate the equicontinuity of the function family\\
$\{(Q_\alpha x)(\cdot): x \in W_r\} \subseteq \mathbb{C}([\zeta,t_2]; Z)$ on interval $[\zeta,t_2]$.

Consider $\zeta \leq t_3 < t_4 \leq t_2$, $x \in W_r$, and $\epsilon > 0$ sufficiently small. Then,
\begin{align*}
    &\left\|Q_\alpha x(t_4) - Q_\alpha x(t_3)\right\| \\
    &\leq \left\|\bigg[\Psi_\alpha(t_4,\zeta) - \Psi_\alpha(t_3,\zeta)\bigg]x_0 \right\| \\
    &\quad + \left\|\int_{\zeta}^{t_3}\bigg[\Psi_\alpha(t_4,s) - \Psi_\alpha(t_3,s)\bigg]
    \bigg(-BH_\alpha (x_0,F)(s) + F(s,x(s))\bigg) \frac{\mathrm{d}s}{s^{1-\alpha}} \right\| \\
    &\quad + \left\|\int_{t_3}^{t_4} \Psi_\alpha(t_4,s) 
    \bigg(-BH_\alpha (x_0,F)(s) + F(s,x(s))\bigg) \frac{\mathrm{d}s}{s^{1-\alpha}} \right\| \\
    \end{align*}
\begin{align*}
    &\leq \left\|\bigg[\Psi_\alpha(t_4,\zeta) - \Psi_\alpha(t_3,\zeta)\bigg]x_0 \right\| \\
    &\quad + \left\|\int_{\zeta}^{t_3 - \epsilon}\bigg[\Psi_\alpha(t_4,s) - \Psi_\alpha(t_3,s)\bigg]
    \bigg(-BH_\alpha (x_0,F)(s) + F(s,x(s))\bigg) \frac{\mathrm{d}s}{s^{1-\alpha}} \right\| \\
    &\quad + \left\|\int_{t_3 - \epsilon}^{t_3}\bigg[\Psi_\alpha(t_4,s) - \Psi_\alpha(t_3,s)\bigg]
    \bigg(-BH_\alpha (x_0,F)(s) + F(s,x(s))\bigg) \frac{\mathrm{d}s}{s^{1-\alpha}} \right\| \\
    &\quad + \left\|\int_{t_3}^{t_4} \Psi_\alpha(t_4,s) 
    \bigg(-BH_\alpha (x_0,F)(s) + F(s,x(s))\bigg) \frac{\mathrm{d}s}{s^{1-\alpha}} \right\| \\
    &:= I_1 + I_2 + I_3 + I_4.
\end{align*}

Certainly, as \( t_4 - t_3 \) approaches zero, it is evident that \( I_1 \rightarrow 0 \), owing to the strong continuity of \( \Psi_\alpha(t, \zeta) \) for \( t \geq\zeta \).\\

We obtain,
\begin{align*}
   & I_3 \leq\\
    & \ 2M  \|B\| \bigg(  (t_3)^{\frac{\alpha-\xi}{1-\xi}}-(t_3- \epsilon)^{\frac{\alpha-\xi}{1-\xi}} \bigg)^{1-\xi} \bigg(\frac{1-\xi}{\alpha-\xi}\bigg)^{1-\xi}  \|H_\alpha \|  
   \bigg( \lVert x_0\rVert +\bigg(  \int_{t_3-\epsilon}^{t_3}  \lVert h_r(s)\rVert^\frac{1}{\xi} \, ds \bigg)^\xi \bigg) \\
    &+2 M \cdot\bigg(  (t_3)^{\frac{\alpha-\xi}{1-\xi}}-(t_3-\epsilon)^{\frac{\alpha-\xi}{1-\xi}} \bigg)^{1-\xi} \bigg(\frac{1-\xi}{\alpha-\xi}\bigg)^{1-\xi} \bigg(\int_{t_3-\epsilon}^{t_3} \lVert h_r(s)\rVert^\frac{1}{\xi} ds \bigg)^\xi.
\end{align*}

\begin{align*}
    I_4 \leq & \ M  \|B\|  \bigg(  (t_4)^{\frac{\alpha-\xi}{1-\xi}}-(t_3)^{\frac{\alpha-\xi}{1-\xi}} \bigg)^{1-\xi} \bigg(\frac{1-\xi}{\alpha-\xi}\bigg)^{1-\xi}\|H_\alpha \|  
    \bigg( \lVert x_0\rVert +\bigg( \int_{t_3}^{t_4}  \lVert h_r(s)\rVert^\frac{1}{\xi} \, ds \bigg)^\xi \bigg) \\
    &+ M \bigg(  (t_4)^{\frac{\alpha-\xi}{1-\xi}}-(t_3)^{\frac{\alpha-\xi}{1-\xi}} \bigg)^{1-\xi} \bigg(\frac{1-\xi}{\alpha-\xi}\bigg)^{1-\xi} \bigg(\int_{t_3}^{t_4} \lVert h_r(s)\rVert^\frac{1}{\xi} ds \bigg)^\xi.
\end{align*}

This demonstrates that \( I_3 \) and \( I_4 \) also tends to zero as \( t_4 - t_3 \) approaches 0, independently of \( x \). Furthermore, Lemma \ref{lemma:Q3} and analogous estimates guarantee that \( I_2 \) approaches 0 as \( t_4 - t_3 \) tends to 0. Consequently, the set \( \{(Q_\alpha x)(\cdot) : x \in W_r\} \) is equicontinuous on \( [\zeta,t_2] \).

\textbf{Step 5:} We confirm that for any $t$ belonging to $[\zeta,t_2]$, the set $\{Q_\alpha (x)(t): x \in W_r\}$ is relatively compact in $Z$.\\

If $t = \zeta$, then $Q_\alpha x(\zeta) = x_0 $, so it holds true for $t = \zeta$.\\

Now, we define
\begin{align*}
    x^{\epsilon}(t)
    &= \Psi_\alpha(t,\zeta)x_{0} - \int_{\zeta}^{t-\epsilon} \Psi_\alpha(t, s)BH_\alpha (x_0,F)(s) \frac{ds}{s^{1-\alpha}} + \int_{\zeta}^{t} \Psi_\alpha(t,s)F(s,x(s)) \frac{ds}{s^{1-\alpha}},
\end{align*}
\hspace{3cm}$\textit{for} \quad t \in \left(\zeta,t_2\right].$\\
Since $\Psi_\alpha(t, s)$ $(t > s \geq \zeta)$ is compact, the set $Q_{\alpha}^{\epsilon}(x)(t) := \{x^{\epsilon}(t) : x \in W_r\}$ is relatively compact in $Z$ for every $\epsilon$, $0 < \epsilon < t \leq t_2$. Furthermore, for every $x \in W_r$,
\begin{align*}
    &\| Q_\alpha(x)(t) - x^\epsilon(t) \|= \left\| \int_{t-\epsilon}^t \Psi_\alpha(t, s) \left[ -BH_\alpha (s, F)(s) + F(s, x(s)) \right] \frac{ds}{s^{1-\alpha}} \right\| \\
    &\leq \int_{t-\epsilon}^t \left\| \Psi_\alpha(t, s) \left[ -BH_\alpha (s, F)(s) + F(s, x(s)) \right] \right\| \frac{ds}{s^{1-\alpha}}\\
    &\leq M \|B\|\cdot \lVert H_\alpha \rVert \bigg(  (t)^{\frac{\alpha-\xi}{1-\xi}}-(t-\epsilon)^{\frac{\alpha-\xi}{1-\xi}} \bigg)^{1-\xi} \bigg(\frac{1-\xi}{\alpha-\xi}\bigg)^{1-\xi}
    \bigg( \lVert x_0\rVert +\bigg(\int_{t-\epsilon}^{t} \lVert h_r(s)\rVert^\frac{1}{\xi} ds \bigg)^\xi \bigg) \\
    &+ M \bigg(  (t)^{\frac{\alpha-\xi}{1-\xi}}-(t-\epsilon)^{\frac{\alpha-\xi}{1-\xi}} \bigg)^{1-\xi} \bigg(\frac{1-\xi}{\alpha-\xi}\bigg)^{1-\xi} \bigg(\int_{t-\epsilon}^{t} \lVert h_r(s)\rVert^\frac{1}{\xi} ds \bigg)^\xi.
\end{align*}

Then, as $\epsilon \to 0^+$, we observe that there are relatively compact sets arbitrarily close to the set $\{ Q_\alpha(x)(t): x \in W_r \}$. Hence, we deduce that $Q_\alpha x(t)$ is relatively compact in $Z$.

Hence, according to the infinite-dimensional version of the Ascoli-Arzelà theorem, \( Q_\alpha \) constitutes a completely continuous operator on \( \mathbb{C}([\zeta,t_2]; Z) \). As a result, according to Schauder's fixed-point theorem, the system (\ref{eq:P}) has at least one fixed point \(x(t)\), which is a mild solution.

\end{proof}

The following existence and uniqueness result for the Cauchy problem (\ref{eq:P}) is based on the Banach contraction principle. We impose the following assumptions:

\begin{itemize}
    \item[($H_3$)]\label{assum:H3_unique} The function \( f(t, s(t)) \) is strongly measurable for any \( x \in \mathbb{C}([\zeta, t_2]; W_r) \) and \( t \in [\zeta, t_2] \).
    
    \item[($H_4$)]\label{assum:H4_unique} There exists a constant \( 0 \leq \xi_1 < \min\{\alpha, \delta\} \) and a function \( \varrho \in L^{\frac{1}{\xi_1}}([\zeta, t_2]; \mathbb{R}^+) \) such that for any \( x, y \in \mathbb{C}([\zeta, t_2]; W_r) \), the following inequality holds:
    \begin{equation*}
        |F(t, x(t)) - F(t, y(t))| \leq \varrho(t) \|x - y\|, \quad t \in [\zeta, t_2].
    \end{equation*}
\end{itemize}

\begin{theorem}\label{thm:Lip}
    If assumptions ($H_1$)-($H_4$) hold, then the Cauchy problem (\ref{eq:P}) has a unique mild solution provided that
    \begin{equation}\label{eq:Eq_Lip}
        MM_1 N < 1,
    \end{equation}
    where $M_1= \| \varrho\|_{L^{\frac{1}{\xi_1}}[\zeta,t_2]}$.
\end{theorem}

\begin{proof}
    From (Step-2), we know that $Q_\alpha$ is an operator mapping $W_r$ to itself. For any $x,y \in W_r$ and $t \in [\zeta,t_2]$, using $(H_4)$, (\ref{eq:N}), and (\ref{Eq:equationQ1}), we obtain
    \begin{align*}
        \lVert (Q_\alpha x)(t) - (Q_\alpha y)(t)\rVert 
        \leq & \lVert \Psi_\alpha(t,\zeta)x_{0} -\Psi_\alpha(t,\zeta)x_{0}\rVert \\
        &+\lVert \int_{\zeta}^{t} \bigg[\Psi_\alpha(t, s)Bu(s) - \Psi_\alpha(t, s)Bu(s) \bigg] \frac{ds}{s^{1-\alpha}}\rVert\\
        &+\lVert \int_{\zeta}^{t_2} \Psi_\alpha(t_2,s)\bigg[ F(s,x(s))- F(s,y(s))\bigg] \frac{ds}{s^{1-\alpha}}\rVert \\
        & \leq  \int_{\zeta}^{t_2} \Psi_\alpha(t_2,s) \varrho(s) \lVert x-y\rVert s^{\alpha -1} ds\\
        &\leq M \bigg( \int_{\zeta}^{t_2} s^{\frac{\alpha-1}{1-\xi_1}} ds  \bigg)^{1-\xi_1} \cdot \lVert \varrho(s) \rVert_{L^{\frac{1}{\xi_1}}[\zeta,t_2]} \lVert x-y \rVert \\
        & \leq M M_1 \bigg( \frac{1-\xi_1}{\alpha -\xi_1}\bigg) \bigg[ (t)^{\frac{\alpha -\xi_1}{1-\xi_1}} -(\zeta)^{\frac{\alpha -\xi_1}{1-\xi_1}} \bigg] \lVert x-y \rVert\\
        &\leq M M_1 N \lVert x-y \rVert.
    \end{align*}

    Thus, we obtain  
    $$ \lVert Q_\alpha x -Q_\alpha y \rVert \leq \big(M M_1 N\big) \lVert x-y \rVert, $$  
    which implies that $Q_\alpha$ is a contraction by (\ref{eq:Eq_Lip}). By applying Banach’s fixed-point principle, we conclude that $Q_\alpha$ has a unique fixed point in $W_r$. This completes the proof.  
\end{proof}

\begin{theorem}\label{thm:th7}
   Given that Assumptions ($H_1$) and ($H_2$) hold, the $\alpha$-conformable differential system (\ref{eq:P}) is exactly null-controllable on the interval $[\zeta, t_2]$ if  
\begin{equation}
M N \gamma \bigg( \lVert B \rVert \lVert H_\alpha  \rVert + 1 \bigg) < 1,  
\end{equation}
and there exists an $r > 0$ such that  
\begin{equation}
\lVert H_\alpha  \rVert \bigg( \lVert x_0 \rVert + \psi(r) \sqrt{t_2 - \zeta} \bigg) \leq r.  
\end{equation}

\end{theorem}

\begin{proof}
    For $x \in \mathbb{C}([\zeta,t_2]; Z)$, and  the control $u$ is determined by the expression:
\begin{equation}
    u(t) = -H_\alpha (x_0 , F)(t) = - (L_{\zeta})^{-1}\left( \Psi_\alpha(t_2,\zeta) x_0+\int_{\zeta}^{t_2} \Psi_\alpha(t_2, s)F(s,x(s))\frac{ds}{s^{1-\alpha}} \right)(t).
    \label{eq:P9}
\end{equation}

Now, it becomes evident that the function \( u(t) \) is properly defined due to \( x_0  \in Z \) and \( F(s, x(s)) \in L^{\frac{1}{\xi}}([\zeta,t_2]; Z)\). Additionally, this control steers from \( x_0 \) to 0.

That is,
\begin{align*}
    &x(t_2,u)=\\
    & \Psi_\alpha(t_2,\zeta)x_{0} + \int_{\zeta}^{t_2} \Psi_\alpha(t_2, s)Bu(s) \frac{ds}{s^{1-\alpha}} + \int_{\zeta}^{t_2} \Psi_\alpha(t_2,s)F(s,x(s)) \frac{ds}{s^{1-\alpha}}\\
    & = \Psi_\alpha(t_2,\zeta)x_{0}\\
    &+ \int_{\zeta}^{t_2} \Psi_\alpha(t, s)B\left[  - (L_{\zeta})^{-1}\left( \Psi_\alpha(t_2,\zeta) x_0+\int_{\zeta}^{t_2} \Psi_\alpha(t_2, s)F(s,x(s))\frac{ds}{s^{1-\alpha}} \right)(t)  \right] \frac{ds}{s^{1-\alpha}}\\
    & + \int_{t_{1}}^{t_2} \Psi_\alpha(t_2,s)F(s,x(s)) \frac{ds}{s^{1-\alpha}},
\end{align*}
It follows easily that,
 $$x(t_2,u)=0.$$
Consequently, system (\ref{eq:P}) is exactly null controllable on $[\zeta,t_2].$
\end{proof}

\section{Application}
\label{sec:Application}
\begin{example}
    As an application of Theorem \ref{thm:th7}, we consider the following partial differential control system:
\begin{align}
    \begin{cases}
        \frac{\partial^{\alpha} u(x, t)}{\partial t^{\alpha}} = \frac{\partial^2 u(x, t)}{\partial x^2} + p(t)u(x, t) + \mu(x, t) + F(t,u(x,t)), \quad 0<\alpha \leq 1,\\
        u(0, t) = u(\pi, t) = 0, \quad t \in [t_1=0,t_2=T], \\
        u(x, \zeta)  = u_0, \quad x \in [0, \pi],\quad  \zeta \in [0,T]\cap [0,(\alpha T)^\frac{1}{\alpha}],
    \end{cases}
    \label{eq:partial_control_system}
\end{align}

where, $ p(\cdot) : [0,(\alpha T)^\frac{1}{\alpha}] \rightarrow \mathbb{R^+}$ is a continuous function.\\
\end{example}

Let \( Y = L^2([0, \pi]) \), and let the operators \( A(t) \) be defined by (see \cite{bib35}, page 131):  
\[
A(t)z = z'' + p(t)z,
\]
with the common domain
\[
D(A(t)) = \{ z(\cdot) \in Y : z, z' \text{ are absolutely continuous}, z'' \in Y, z(0) = z(\pi) = 0 \}.
\]

We now verify that \( A(t) \) generates an evolution operator \( \Psi_\alpha(t, s) \), which satisfies assumptions (\ref{app: a} through (\ref{app: d}. The evolution operator \( \Psi_\alpha(t, s) \) is given by
\[
\Psi_\alpha(t, s) = S_s(t-s)e^{-\int_\frac{s^\alpha}{\alpha}^\frac{t^\alpha}{\alpha} p(\tau) d(\tau, \alpha)},
\]
where \( S_s(t) \), for \( t \geq 0 \), is the compact \( C_0 \), \(\alpha\)-conformable semigroup generated by the operator \( A \) with

$$Az = z'',$$ 

for \( z \in D(A(t)) \), it is straightforward to verify that \( A \) has a discrete spectrum with eigenvalues \( -n^2 \) for \( n \in \mathbb{N} \), and corresponding normalized eigenvectors \( e_n(x) = \sqrt{2} \sin(n x) \). Consequently, for \( z \in D(A(t)) \) (see \cite{bib35}, page 29), we have:
\[ A(t)z = \sum_{n=1}^\infty (-n^2 - p(t))\langle z,e_n\rangle e_n, \]
The common domain coincides with that of the operator $A$. Moreover, for each $z \in Y$, the expression for $\Psi_\alpha(t, s)z$ is given by
\[ \Psi_\alpha(t, s)z = \sum_{n=1}^\infty e^{-n^2(\frac{t^\alpha}{\alpha}-\frac{s^\alpha}{\alpha})-\int_\frac{s^\alpha}{\alpha}^\frac{t^\alpha}{\alpha} p(\tau)\,d(\tau,\alpha)} \langle z, e_n \rangle e_n. \]
Therefore, for every $z \in Y,$
\begin{align*}
   & \Psi_\alpha(t, s)z =\Psi_\alpha^*(t, s)z\\
    &= \sum_{n=1}^\infty e^{-n^2(\frac{t^\alpha}{\alpha}-\frac{s^\alpha}{\alpha})-\int_\frac{s^\alpha}{\alpha}^\frac{s^\alpha}{\alpha} p(\tau)\,d(\tau,\alpha)} \sin(nx) \int_{0}^{\pi} z(\beta) \sin(n\beta) \, d\beta, \quad z \in Y,
\end{align*}
For the function \( F(\cdot, \cdot) : [t_1, t_2] \times \mathbb{R} \to \mathbb{R} \), we make the following assumptions:

(i) For any \( x \in \mathbb{R} \), \( F(\cdot, x) \) is measurable on \( [t_1=0, t_2=T] \).

(ii) For any fixed \( t \in [t_1,t_2] \), \( F(t, \cdot) \) is continuous, and there exists a constant \( c \geq 0 \) such that
\[ |F(t, x)| \leq c|x| ,\quad \textit{for all} \quad t \in [t_1,t_2] . \]
Under these assumptions, the function \( F(t, w) \) also satisfies the Assumption ($H_1$).\\
Now, let \( \mu \in L^2([0, T]; Y) \), and consider \( B = I \), which implies \( B^* = I \). Consequently, the system
\eqref{eq:partial_control_system} transforms into the system \eqref{eq:P}. Next, we analyze the resulting linear system with an additional term \( h \in L^{\frac{1}{\xi}}([0, T], Y) \):

\begin{equation}
\begin{cases}
\frac{\partial^{\alpha} z(x, t)}{\partial t^{\alpha}} = \frac{\partial^2 z}{\partial x^2}(x, t) + p(t)z(x, t) + \mu(x, t) + h(x,t), \quad \alpha \in (0,1] \\
z(0, t) = z(\pi, t) = 0, \quad t \in [t_1=0,t_2=T],\\
z(x, \zeta) = z_0, \quad x \in [0, \pi] ,\quad  \zeta \in [0,T]\cap [0,(\alpha T)^\frac{1}{\alpha}],
\end{cases}
\label{eq:4.2}
\end{equation}

By virtue of Remark \ref{remark1}, the exact null controllability of the linear system~\eqref{eq:4.2} is equivalent to the existence of $\gamma > 0$ such that

$$\int_{\zeta}^{T} \|B^*\Psi_\alpha^*(T,s)z\|^2 d(s,\alpha) \geq \gamma \left(\| \Psi_\alpha^*(T,\zeta)z\|^2 + \int_{\zeta}^{T} \|\Psi_\alpha^*(T,s)z\|^2 d(s,\alpha) \right),
$$

or equivalently,

\begin{equation}
    \int_{\zeta}^{T} \|\Psi_\alpha(T,s)z\|^2 d(s,\alpha) \geq \gamma \left(\| \Psi_\alpha(T,\zeta)z\|^2 + \int_{\zeta}^{T} \|\Psi_\alpha(T,s)z\|^2 d(s,\alpha) \right).
    \label{eq:4.3}
\end{equation}

In \cite{bib35}, it is demonstrated that the linear system [\ref{eq:4.2}] with \( h = 0 \) is exactly null controllable if

$$  \int_{\zeta}^{T} \|\Psi_\alpha(T,s)z\|^2 d(s,\alpha) \geq (T-\zeta) \| \Psi_\alpha(T,\zeta)z\|^2$$

Hence,
$$ \frac{1}{1+(T-\zeta)} \int_{\zeta}^{T} \|\Psi_\alpha(T,s)z\|^2 d(s,\alpha) \geq \frac{(T-\zeta)}{1+(T-\zeta)} \| \Psi_\alpha(T,\zeta)z\|^2$$

or,
$$ \left(\frac{(T-\zeta)+1-(T-\zeta)}{1+(T-\zeta)}\right) \int_{\zeta}^{T} \|\Psi_\alpha(T,s)z\|^2 d(s,\alpha) \geq \frac{(T-\zeta)}{1+(T-\zeta)} \| \Psi_\alpha(T,\zeta)z\|^2$$

or,
$$   \int_{\zeta}^{T} \|\Psi_\alpha(T,s)z\|^2 d(s,\alpha) \geq \frac{(T-\zeta)}{1+(T-\zeta)} \left(\| \Psi_\alpha(T,\zeta)z\|^2 + \int_{\zeta}^{T} \|\Psi_\alpha(T,s)z\|^2 d(s,\alpha) \right).$$

Hence, when \(\gamma = \frac{(T - \zeta)}{(T - \zeta) + 1}\), equation (\ref{eq:4.3}) remains valid. As a result, the linear system described by Equation (\ref{eq:4.2}) exhibits exact null controllability over the interval \([\zeta, T]\). Consequently, Theorem \ref{thm:th7} implies that the system in equation (\ref{eq:partial_control_system}) is exactly null controllable over the range \([\zeta, T]\), provided that the condition in inequality (\ref{eq:P8}) is satisfied.

\begin{remark}
  An example of a semilinear function that satisfies assumptions ($H_1$)  is 
\[
F(t, u(x, t)) = \sin(x) + (t)^{2\alpha} \cos(2x),
\]
and $p(t) = t^3 + t^2 + 1 $, then  the system (\ref{eq:partial_control_system}) becomes exactly null controllable.

\end{remark}

 \bibliographystyle{elsarticle-num} 
 \bibliography{cas-refs}

\begin{thebibliography}{10}
\expandafter\ifx\csname url\endcsname\relax
  \def\url#1{\texttt{#1}}\fi
\expandafter\ifx\csname urlprefix\endcsname\relax\def\urlprefix{URL }\fi
\expandafter\ifx\csname href\endcsname\relax
  \def\href#1#2{#2} \def\path#1{#1}\fi

\bibitem{bib17}
H.~Beyer, S.~Kempfle, Definition of physically consistent damping laws with fractional derivatives, ZAMM-Journal of Applied Mathematics and Mechanics/Zeitschrift f{\"u}r Angewandte Mathematik und Mechanik 75~(8) (1995) 623--635.

\bibitem{bib18}
M.~Caputo, Linear models of dissipation whose q is almost frequency independent—ii, Geophysical journal international 13~(5) (1967) 529--539.

\bibitem{bib19}
J.~He, Some applications of nonlinear fractional differential equations and their approximations, Bull. Sci. Technol 15~(2) (1999) 86--90.

\bibitem{bib20}
J.-H. He, Approximate analytical solution for seepage flow with fractional derivatives in porous media, Computer methods in applied mechanics and engineering 167~(1-2) (1998) 57--68.

\bibitem{bib21}
B.~Mathieu, P.~Melchior, A.~Oustaloup, C.~Ceyral, Fractional differentiation for edge detection, Signal Processing 83~(11) (2003) 2421--2432.

\bibitem{bib22}
M.~F. Silva, J.~T. Machado, A.~Lopes, Fractional order control of a hexapod robot, Nonlinear Dynamics 38 (2004) 417--433.

\bibitem{bib23}
J.~Yan, C.~Li, On chaos synchronization of fractional differential equations, Chaos, Solitons \& Fractals 32~(2) (2007) 725--735.

\bibitem{bib24}
A.~A. Kilbas, H.~M. Srivastava, J.~J. Trujillo, Theory and Applications of Fractional Differential Equations, Vol. 204, Elsevier, 2006.

\bibitem{bib25}
K.~S. Miller, B.~Ross, An introduction to the fractional calculus and fractional differential equations, (No Title) (1993).

\bibitem{bib26}
I.~Podlubny, Fractional differential equations: an introduction to fractional derivatives, fractional differential equations, to methods of their solution and some of their applications, elsevier, 1998.

\bibitem{bib1}
R.~Khalil, M.~Al~Horani, A.~Yousef, M.~Sababheh, A new definition of fractional derivative, Journal of computational and applied mathematics 264 (2014) 65--70.

\bibitem{bib2}
T.~Abdeljawad, On conformable fractional calculus, Journal of computational and Applied Mathematics 279 (2015) 57--66.

\bibitem{bib16}
M.~A. Hammad, R.~Khalil, Conformable fractional heat differential equation, Int. J. Pure Appl. Math 94~(2) (2014) 215--221.

\bibitem{bib27}
M.~Abu~Hammad, R.~Khalil, Abel’s formula and wronskian for conformable fractional differential equations, Int. J. Differ. Equ. Appl. 13(3) (2014) 177--183.

\bibitem{bib28}
A.~Atangana, D.~Baleanu, A.~Alsaedi, New properties of conformable derivative, Open Mathematics 13~(1) (2015) 000010151520150081.

\bibitem{bib29}
O.~T. Birgani, S.~Chandok, N.~Dedovic, S.~Radenovic, A note on some recent results of the conformable fractional derivative, Advances in the Theory of Nonlinear Analysis and its Application 3~(1) (2019) 11--17.

\bibitem{bib30}
E.~{\"U}nal, A.~G{\"o}kdo{\u{g}}an, Solution of conformable fractional ordinary differential equations via differential transform method, Optik 128 (2017) 264--273.

\bibitem{engel2000one}
K.-J. Engel, R.~Nagel, S.~Brendle, One-parameter semigroups for linear evolution equations, Vol. 194, Springer, 2000.

\bibitem{bib34}
A.~Pazy, Semigroups of linear operators and applications to partial differential equations, Vol.~44, Springer Science \& Business Media, 2012.

\bibitem{goldstein2017semigroups}
J.~A. Goldstein, Semigroups of linear operators and applications, Courier Dover Publications, 2017.

\bibitem{bib36}
J.~Dauer, N.~Mahmudov, Exact null controllability of semilinear integrodifferential systems in hilbert spaces, Journal of mathematical analysis and applications 299~(2) (2004) 322--332.

\bibitem{bib3}
V.~Bryant, A remark on a fixed-point theorem for iterated mappings, American Mathematical Monthly 75 (1968) 399--400.

\bibitem{bib4}
L.~Sadek, B.~Abouzaid, E.~M. Sadek, H.~T. Alaoui, Controllability, observability and fractional linear-quadratic problem for fractional linear systems with conformable fractional derivatives and some applications, International Journal of Dynamics and Control 11~(1) (2023) 214--228.

\bibitem{bib5}
N.~I. Mahmudov, M.~Ayd{\i}n, Representation of solutions of nonhomogeneous conformable fractional delay differential equations, Chaos, Solitons \& Fractals 150 (2021) 111190.

\bibitem{Sarikaya}
M.~Z. Sar{\i}kaya, A.~Akkurt, H.~Budak, M.~E. T{\"u}rkay, H.~Yildirim, On some special functions for conformable fractional integrals, Konuralp Journal of Mathematics 8~(2) (2020) 376--383.

\bibitem{bib35}
R.~F. Curtain, H.~Zwart, An introduction to infinite-dimensional linear systems theory, Vol.~21, Springer Science \& Business Media, 2012.

\end{thebibliography}





\end{document}